\documentclass[12pt]{amsart}
\usepackage{amsmath, amssymb, amsthm, latexsym}
\usepackage{amssymb,amscd,amsmath}
\usepackage{latexsym}
\usepackage[center]{caption}
\usepackage{tikz}
\newcounter{braid}
\newcounter{strands}

\DeclareMathAlphabet{\bsf}{OT1}{cmss}{bx}{n}

\pgfkeyssetvalue{/tikz/braid height}{1cm}
\pgfkeyssetvalue{/tikz/braid width}{1cm}
\pgfkeyssetvalue{/tikz/braid start}{(0,0)}
\pgfkeyssetvalue{/tikz/braid colour}{black}
\pgfkeys{/tikz/strands/.code={\setcounter{strands}{#1}}}

\makeatletter
\def\cross{%
  \@ifnextchar^{\message{Got sup}\cross@sup}{\cross@sub}}

\def\cross@sup^#1_#2{\render@cross{#2}{#1}}

\def\cross@sub_#1{\@ifnextchar^{\cross@@sub{#1}}{\render@cross{#1}{1}}}

\def\cross@@sub#1^#2{\render@cross{#1}{#2}}

\def\render@cross#1#2{
  \def\strand{#1}
  \def\crossing{#2}
  \pgfmathsetmacro{\cross@y}{-\value{braid}*\braid@h}
  \pgfmathtruncatemacro{\nextstrand}{#1+1}
  \foreach \thread in {1,...,\value{strands}}
  {
    \pgfmathsetmacro{\strand@x}{\thread * \braid@w}
    \ifnum\thread=\strand
    \pgfmathsetmacro{\over@x}{\strand * \braid@w + .5*(1 - \crossing) * \braid@w}
    \pgfmathsetmacro{\under@x}{\strand * \braid@w + .5*(1 + \crossing) * \braid@w}
    \draw[braid] \pgfkeysvalueof{/tikz/braid start} +(\under@x pt,\cross@y pt) to[out=-90,in=90] +(\over@x pt,\cross@y pt -\braid@h);
    \draw[braid] \pgfkeysvalueof{/tikz/braid start} +(\over@x pt,\cross@y pt) to[out=-90,in=90] +(\under@x pt,\cross@y pt -\braid@h);
    \else
    \ifnum\thread=\nextstrand
    \else
     \draw[braid] \pgfkeysvalueof{/tikz/braid start} ++(\strand@x pt,\cross@y pt) -- ++(0,-\braid@h);
    \fi
   \fi
  }
  \stepcounter{braid}
}

\tikzset{braid/.style={double=\pgfkeysvalueof{/tikz/braid colour},double distance=1pt,line width=2pt,white}}

\newcommand{\braid}[2][]{%
  \begingroup
  \pgfkeys{/tikz/strands=2}
  \tikzset{#1}
  \pgfkeysgetvalue{/tikz/braid width}{\braid@w}
  \pgfkeysgetvalue{/tikz/braid height}{\braid@h}
  \setcounter{braid}{0}
  \let\sigma=\cross
  #2
  \endgroup
}
\makeatother

\input xypic
\newtheorem{theorem}{Theorem}
\newtheorem{proposition}[theorem]{Proposition}

\newtheorem{lemma}[theorem]{Lemma}

\newtheorem{corollary}[theorem]{Corollary}


\def\Z{\mathbb{Z}}

\def\Pi{\mathbb{P}^{\infty}}

\def\qed{\hfill$\square$\medskip}

\def\Zpk{\mathbb{Z}/p^{k}}
\def\Zpk1{\mathbb{Z}/p^{k-1}}
\def\k{\kappa}

\newcommand{\rref}[1]{(\ref{#1})}

\newcommand{\beg}[2]{\begin{equation}\label{#1}#2\end{equation}}
\def\r{\rightarrow}

\def\sl2{\widetilde{SL_{2}(\Z)}}

\title[$RO(G)$-graded ordinary cohomology]{On $RO(G)$-graded equivariant ``ordinary" cohomology where
$G$ is a power of $\Z/2$}
\author{John Holler and Igor Kriz}

\begin{document}

\begin{abstract}
We compute the complete $RO(G)$-graded coefficients of ``ordinary" cohomology with coefficients in $\Z/2$ for $G=(\Z/2)^n$.
\end{abstract}

\maketitle

\section{Introduction}

The notion of a cohomology theory graded by elements of the real representation ring ($RO(G)$-graded cohomology) is a
key concept of equivariant stable homotopy theory of a finite or compact Lie group $G$. Like much of stable homotopy 
theory, perhaps one of the first known example was K-theory. Atiyah and Singer \cite{as} introduced equivariant K-theory of a compact Lie group $G$ and proved that it is naturally $RO(G)$-graded. In fact, Bott periodicity identifies 
many
of the ``dimensions" in $RO(G)$, and relates others to "twistings" (see Karoubi \cite{kar} and, for a more recent
treatment, Freed, Hopkins and Teleman \cite{fht}). Pioneered by Adams and Greenlees \cite{green}, the general $RO(G)$-graded stable homotopy theory theory found firm foundations in the fundamental book of Lewis, May and Steinberger 
\cite{lms}. 

\vspace{3mm}

Despite the clear importance of the concept, beyond K-theory, calculations of $RO(G)$-graded cohomology are few and far in between. Perhaps the most striking case is ``ordinary" $RO(G)$-graded cohomology. Bredon \cite{bredon}
discovered $\Z$-graded $G$-equivariant cohomology associated with a {\em coefficient system} which is ``ordinary" in the sense that the cohomology of a point is concentrated in a single dimension. It was later discovered (\cite{lmm})
that such a theory becomes $RO(G)$-graded when the coefficient system enjoys the structure of a 
{\em Mackey functor} \cite{dress}, which means that it allows building in an appropriate concept of {\em transfer}.
Strikingly, the $RO(G)$-graded coefficients were not known in any single non-trivial case.

\vspace{3mm}

Complete calculations of $RO(\Z/2)$-graded coefficients, however, are important in Real-oriented stable homotopy theory,
because they exhibit the analogy with the complex-oriented case. Real orientation was, once again, 
discovered first by Atiyah
in the case of $K$-theory \cite{atiyah}, and was subsequently extended to cobordism by Landweber \cite{land}.
$RO(\Z/2)$-graded cohomology with coefficients in the Burnside ring Mackey functor was calculated by Stong \cite{sgl}.
A systematic pursuit of real-oriented homotopy theory was started by Araki \cite{araki}, and developed further
by Hu and Kriz \cite{hk} with many calculations, including a complete calculation of the $RO(G)$-graded coefficients
of Landweber's Real cobordism spectrum. In the process, \cite{hk} also calculated the $RO(\Z/2)$-graded 
ordinary cohomology of the ``constant" Mackey functors $\Z$ and $\Z/2$ (i.e. the Mackey functors uniquely 
extending the constant coefficient systems). A major development was the work of Hill, Hopkins and Ravenel
\cite{hhr}, who partially extended the calculations of \cite{hk} to $\Z/(2^k)$ (with special interest in $k=3$),
and applied this to solving the Kervaire-Milnor problem by showing the non-existence of manifolds of Kervaire invariant
$1$ in dimensions $>126$. A still more complete calculation of $RO(G)$-graded ordinary equivariant cohomology
of the constant Mackey functors for $G=\Z/(2^k)$ was more recently given in \cite{hkem}.

\vspace{3mm}

Still, no calculations of $RO(G)$-graded cohomology beyond K-theory were known for groups other than
where $G$ is a primary cyclic group. In a spin-off \cite{hkherm} of their joint solution with Ormsby
\cite{hko} of Thomason's homotopy limit problem for Hermitian K-theory, Hu and Kriz 
computed the $RO(G)$-graded coefficients of topological Hermitian cobordism, which has $G=\Z/2\times \Z/2$.
However, this is a rather special case, where many periodicities occur.

\vspace{3mm}

The purpose of the present paper is to calculate the $RO(G)$-graded coefficients of the ordinary equivariant
cohomology of the ``constant" $\Z/2$ Mackey functor for $G=(\Z/2)^n$. There are several reasons to focus on this
case. The group $(\Z/2)^n$ has an exceptionally simply described real representation ring, thus eliminating the need 
to handle representation-theoretical exceptions such as distinguishing between real and complex (let alone, quaternionic)
representations. The coefficients $\Z/2$ are more convenient than $\Z$, since they eliminate the need to
consider extensions. Despite all this, the complete answer is complicated, however, and in general, we are only able
to present it in the form of the cohomology of an $n$-stage chain complex.

\vspace{3mm}

Our method is based on {\em isotropy separation}, a term coined by Greenlees and May \cite{gmsur},
to mean considering separately the contributions of subgroups of $G$. An isotropy separation spectral sequence
was developed in \cite{abk}, but we use a different spectral sequence here. The reason is that in \cite{abk},
we are not concerned with $RO(G)$-graded coefficients, but rather with computing the complete $\Z$-graded coefficients 
of equivariant complex cobordism of a finite abelian group $G$ as a ring. Based on generalizing the method
of \cite{kriz} in the case of $G=\Z/p$, in the case of $\Z$-graded equivariant complex cobordism, 
one can set up a spectral sequence of rings which collapses to $E^2$ in a single filtration degree. This
means that the complete ring structure can be recovered, which is a special property of complex cobordism.
It is worth mentioning that the spectral sequence of \cite{abk} contains many ``completed'' (i.e., for example, uncountable) terms.

\vspace{3mm}

The case of ordinary $RO(G)$-graded equivariant cohomology is quite different, however, in that 
the spectral sequence fails to collapse to a single degree. Even for $G=\Z/p$, we observe that a part of the
coefficients are in filtration degree $0$ and a part in filtration degree $-1$ (graded homologically). This caused us
to give up, at least for now, calculating the complete ring structure, and use a spectral sequence which is
more amenable to calculations instead.

\vspace{3mm}

Another key ingredient in our computation is the concept of {\em geometric fixed points} of an $RO(G)$-graded
equivariant cohomology theory. This concept was introduced (using a different terminology)
by tom Dieck \cite{td}, who calculated the geometric 
fixed points of equivariant complex cobordism. As far as we know, the term geometric fixed points was coined
by Greenlees and May, and is recorded in \cite{lms}.
Unlike actual fixed points, the geometric fixed point coefficients are {\em periodic} with respect to all
non-trivial irreducible real representations of $G$. Thus, instead of $RO(G)$-graded, the geometric fixed points
are, again, only $\Z$-graded. This is a big advantage in expressing the answer. Note that the ring 
$RO((\Z/2)^n)$ is huge: it is the free abelian group on $2^n$ generators! On the downside, 
the term ``geometric" fails to carry the expected implications in the case of ordinary equivariant cohomology:
we know of no geometry that would help calculating them. Still, in the case $G=(\Z/2)^n$, a complete calculation
of the geometric fixed point ring of $H\Z/2$ is possible using spectral sequence methods. This is our Theorem
\ref{t1}.

\vspace{3mm}

The main method of this paper is, basically, setting up another spectral sequence which enables the calculation
of the coefficients of $H\Z/2_{(\Z/2)^n}$ by investigating how they differ from the coefficients of the geometric
fixed points. There results a spectral sequence, which, in a fairly substantial range of
$RO(G)$-graded dimensions, collapses to $E^2$ in degree $0$. More precisely, the range is, graded homologically,
suspensions by elements of $RO(G)$ where summands of non-trivial irreducible representations
occur with {\em non-positive} coefficients. Alternately, graded cohomologically, this is the range of
suspensions by actual representations, possibly minus a trivial representation. (As it turns out, however, 
in this case, when the trivial representation has a negative coefficient, the cohomology group is $0$.)
In this case, we can both recover the complete ring structure, since the ring embeds into the ring of
geometric fixed points tensored with $RO(G)$. We also have a nice concise formula for the Poincare series
in this case (Theorem \ref{t2}).

\vspace{3mm}

In the case of completely general $RO(G)$-dimension with $G=(\Z/2)^n$, 
we are only able to give a spectral sequence in $n$
filtration degrees, which collapses to $E^2$ and calculates the $RO(G)$-graded coefficient group of
$H\Z/2_G$. Thus, this gives an algebraically defined chain complex whose homology are the desired groups
(Theorem \ref{t3}). We give an example of a complete calculation of the Poincare series of the $RO(G)$-graded
coefficients of $H\Z/2_{\Z/2\times\Z/2}$ (the case $n=2$), which clearly shows that the answer gets 
complicated, and additional complications arise for $n\geq 3$.

\vspace{3mm}

The present paper is organized as follows: In Section \ref{s2}, we introduce the necessary conventions and
notation. In Section \ref{s3}, we compute the geometric fixed points. In Section \ref{s4}, we compute the
coefficients in dimensions involving elements of $RO(G)$ where non-trivial irreducible representations have
non-positive coefficients (graded homologically). In Section \ref{s5}, we calculate the chain complex computing the
complete $RO(G)$-graded coefficients of $H\Z/2_G$ for $G=(\Z/2)^n$. In Section \ref{s6}, we treat the
example of $n=2$. The authors apologize to the readers for not stating their theorems in the Introduction.
Even in the prettiest cases, the theorems involve quite a lot of notation and technical prerequisites. We prefer
to state them properly in the text.

\vspace{3mm}
\noindent
{\bf Recent developments: Odd primes, and hyperplane arrangements.} 
While this paper was under review, several developments took place. A generalization
of the present result to $(\Z/p)^n$ for $p$ an odd prime was found by Holler. The authors also
found out that the ring described in Theorem 
\ref{t1} is a previously known object in algebraic geometry, related to a certain compactification of complements of
hyperplane arrangements referred to as {\em the reciprocal plane}. 

More concretely, for a set 
$S=\{z_\alpha\}$ of equations
of hyperplanes through $0$ in an affine space $Spec(F[u_1,\dots,u_n])$ of a field $F$, one considers
the subring $R_S$ of 
\beg{efff1}{(\prod_{\alpha\in S} z_\alpha)^{-1}F[u_1,\dots,u_n]}
generated by the elements $z_\alpha^{-1}$ (which correspond to our elements $x_\alpha$).
The ring was first described by 
Terao \cite{terao}, and a particularly nice presentation was found by Proudfoot and Speyer \cite{ps}. In the
case of an odd prime $p$, one deals analogously with the subring $\Xi_S$ of 
\beg{efff2}{(\prod_{\alpha\in S} z_\alpha)^{-1}F[u_1,\dots,u_n]\otimes_F \Lambda(du_1,\dots,du_n)}
generated by $z_\alpha^{-1}$ and $d\log(z_\alpha)$, which are topologically in dimensions $2$ and $1$,
respectively. The
analogues of the constructions of \cite{ps, terao} in this graded-commutative case, 
and the reciprocal plane compactification, were recently
worked out by S. Kriz \cite{sk}. 

Our emphasis is quite different form the authors of \cite{ps, terao}, who, doing classical algebraic geometry, were
mostly interested in characteristic $0$. Their
arguments, however, work in general. The ring described in Theorem \ref{t1} (and its
$\Z/p$ analogue discovered by Holler, i.e. the geometric fixed point ring of
$H\Z/p_{G}$ where $G=(\Z/p)^n$) is related to the hyperplane arrangement
of {\em all} hyperplanes through $0$ in the the $n$-dimensional affine space over $\Z/p$. It follows, however, from
the description of \cite{ps, terao, sk} that for a subset $S^\prime$ of a hyperplane arrangement $S$,
the ring $R_{S^\prime}$ (resp. $\Xi_{S^\prime}$) is a subring of $R_{S}$ (resp. $\Xi_{S}$).
It follows in turn that for {\em every} hyperplane arrangement in $G=(\Z/p)^n$, the
$\Z$-graded part of the coefficient ring of the spectrum
$$\bigwedge_{\alpha\in S} S^{\infty\alpha}\wedge H\Z/p_G$$
is $R_S$ for $p=2$, and $\Xi_S$ for and odd prime $p$.

\vspace{5mm}

\section{Conventions and notation}\label{s2}
Throughought this paper, let $G=(\Z/2)^n$. Then the real representation ring of $G$ is canonically identified as
$$RO(G)=\Z[G^*]$$
where $G^*=Hom(G,\Z/2)$. Recall \cite{lms} that for $H\subseteq G$, we have the family $\mathcal{F}[H]$ consisting of all subgroups $K\subset G$ with $H\nsubseteq K$. (In the case of $H=G$, $\mathcal{F}[G]$ is simply the
family $\mathcal{P}$ of proper subgroups of $G$.)
Recall further that for any family $\mathcal{F}$ (a set of subgroups of $G$ closed under subconjugation, which is the same as closed under subgroups, as $G$ is commutative), we have a cofibration sequence
$$E\mathcal{F}_+\r S^0 \r \widetilde{E\mathcal{F}}$$
where $E\mathcal{F}$ is a $G$-CW-complex whose $K$-fixed point set is contractible when $K\in\mathcal{F}$
and empty otherwise. For our choice of $G$, we may then choose a model
\beg{egeom+}{\widetilde{E\mathcal{F}[H]}=\bigwedge_{\alpha\in G^*, \;\alpha|H\neq 0}S^{\infty\alpha}.
}
Here $S^{\infty\alpha}$ is the direct limit of $S^{n\alpha}$ with respect to the inclusions
\beg{egeom*}{S^0\r S^\alpha
}
given by sending the non-base point to $0$. The other construction we use is the family $\mathcal{F}(H)$ 
of all subgroups of a subgroup $H\subseteq G$. We will write simply
$$EG/H=E\mathcal{F}(H).$$
The cardinality of a finite set $S$ will be denoted by $|S|$. We will also adopt a convention from 
\cite{hk} where, for an $RO(G)$-graded spectrum $E$, $E_*$ denotes the $\Z$-indexed coefficients (=homotopy groups)
of $E$, while the $RO(G)$-indexed coefficients will be denoted by $E_\star$.
As is customary, we will also denote by $S(V)$ the unit sphere of a representation $V$, while by $S^V$ we denote the
$1$-point compactification of $V$. The $RO(G)$-graded dimension of a homogeneous element $x\in E_\star$ will be
denoted by $|x|$.

\vspace{3mm}

\section{The geometric fixed points}\label{s3}

In this section, we compute the coefficients of the geometric fixed point spectrum $\Phi^GH\Z/2$. We have
\beg{egeom1+}{\Phi^GH\Z/2=(\widetilde{E\mathcal{F}[G]}\wedge H\Z/2)^G.}
By \rref{egeom+}, suspension of $H\Z/2$
by any non-trivial irreducible real representation of $G$ gives an isomorphism on coefficients, so the coefficients $(\Phi^G\Sigma^?H\Z/2)_*$ are only $\Z$-graded, not $RO(G)$-graded. More specifically,
we have a cofibration sequence
\beg{egeom++}{EG/Ker(\alpha)_+\r S^0\r S^{\infty\alpha},}
so smashing over all non-trivial $1$-dimensional representations $\alpha$, using \rref{egeom+}, we may represent 
$$\widetilde{E\mathcal{F}[G]}\wedge H\Z/2$$
as the iterated cofiber of a $2^n-1$-dimensional cube of the form
\beg{egeom1}{H\Z/2\wedge\bigwedge_{0\neq \alpha\in G^*}(EG/Ker(\alpha)_+\r S^0).
}
Taking coefficients in \rref{egeom1} then gives a spectral sequence converging to $\Phi^GH\Z/2_*$. Now also note that
\beg{egeom2}{EG/H_1\times\dots \times EG/H_k\simeq EG/(H_1\cap\dots\cap H_k).
}
From this, we can calculate the spectral sequence associated with the iterated cofiber of the cube \rref{egeom1}.
Let us grade the spectral sequence homologically, so the term $H\Z/2_*=\Z/2$ is in $E^{1}_{0,0}$. 
The rest of the $E^1$-term is then given as
\beg{egeom2a}{E^{1}_{p,*}=\bigoplus_{S\in \mathcal{S}_p} Sym_{\Z/2}((G/\cap\{Ker
(\alpha)\mid \alpha\in S\})^*)\cdot y_S
}
where $\mathcal{S}_p$ is the set of all subsets of $G^*\smallsetminus\{0\}$ 
of cardinality $p$. (The last factor $y_S$ of \rref{egeom2a}
is only a generator written to distinguish the summands.) Now the $E^2$-term can also be calculated using the 
following

\begin{lemma}
\label{l1}
Consider the differential $\partial$ on
$$Q_n=\Z/2\{y_S\mid S\subseteq (\Z/2)^n\smallsetminus \{0\}\}$$
given by
\beg{epartial}{\partial(y_S)=\sum_{s\in S,\;\langle {S\smallsetminus\{s\}}\rangle=\langle S\rangle} 
y_{S\smallsetminus \{s\}}.}
Then the homology is the $\Z/2$-vector space (freely) generated by a set $F_n$ described inductively as follows:
$$F_1=\{y_\emptyset, y_{\{(1)\}}\},$$
$$F_n=F_{n-1}\cup \{y_{S\cup \{x\}}\mid S\in F_{n-1}, x\in (\Z/2)^{n-1}\times\{1\}\}.$$
In other words, $F_n$ consists of the basis elements $y_S$
where $S$ are all the $\Z/2$-linearly independent (in $G^*$) subsets in (not necessarily reduced) 
row echelon form with respect to reversed order of columns (so the first pivot is in the last possible column etc.).
\end{lemma}

\begin{proof}
Consider a differential on $Q_n$ given by
\beg{eddd1}{d(y_S)=\sum_{s\in S} y_{S\smallsetminus \{s\}}.}
Then the homology is $0$ for $n>0$ and $\Z/2$ for $n=0$. Now consider
an increasing filtration on $Q_n$ by making the filtration degree $\gamma(S)$ of a basis element $y_S$ equal to
$rank\langle S\rangle$, the rank of the $\Z/2$-vector space generated by $S$.
Then the $E^1$-term is what we are trying to calculate. 

On the other hand, in the answer $C=\Z/2(F_n)$ suggested in the statement of the Lemma
(which, note, consists of elements of $E^1$), 
the formula for $d^1$ is the same as the formula \rref{eddd1} 
for $d$. We claim that
\beg{epartial1}{H_*(C,d)=0.}
To see this, note that for any fixed non-empty set $S$ in row echelon form, the subcomplex $C_S$ generated by
$y_{S^\prime}$ subsets of $S^\prime\subseteq S$
is just a tensor product of copies of 
\beg{egeom2b}{\diagram \Z/2\rto^\cong & \Z/2,\enddiagram}
and hence satisfies 
$$H_*(C_S,d)=0.$$
On the other hand, $C$ for $n>0$
is a sum of the complexes $C_S$ where $S$ ranges over {\em maximal} linearly independent subsets of $(\Z/2)^n$ in REF (i.e. those which have exactly $n$ elements), while the intersection of any subset of those 
complexes 
$$C_{S_1\cap\dots \cap S_k}=C_{S_1}\cap\dots\cap C_{S_k}$$
has zero homology because
$$(1,0,\dots,0)\in S_1\cap\dots \cap S_k$$
and hence $S_1\cap\dots\cap S_k\neq \emptyset$. This implies \rref{epartial1}. 

Now the statement follows by 
induction on $n$ using comparison theorems for spectral sequences. More concretely, if we denote
by $C^\prime\subset C$ the subcomplex generated by linearly independent subsets $S$ with $|S|<n$, and
$Q^\prime\subset Q_n$ the subcomplex generated by sets $S$ which span a subspace of dimension $<n$, then
the induction hypothesis (given that an intersection of vector subspaces is a vector subspace), shows that
the embedding $C\subset Q_n$ restricts to a quasi-isomorphism 
\beg{eccc}{C^\prime \subset Q^\prime.}
Since the
homologies of both $C$ and $Q_n$ are $0$, we see that the homomorphism on degree $n$ subcomplexes must
induce an isomorphism on homology, thus implying that the degree $n$ part of the $E^1$ term of our spectral
sequence for $Q_n$ is just the degree $n$ part of $C$ (which is, of course, isomorphic to $\Z/2$).
\end{proof}

\vspace{3mm}

Now by Lemma \ref{l1}, the $E^2$-term of the spectral sequence of the cube \rref{egeom1} is
\beg{egeom2aa}{E^2=\bigoplus_{S\in F_n}Sym_{\Z/2}((G/\cap\{Ker
(\alpha)\mid \alpha\in S\})^*)\cdot y_S
}
(where we identify $G^*\cong (\Z/2)^n$).

Now consider, for $0\neq \alpha:G\r\Z/2$, the map 
\beg{egeommap1}{f_\alpha:\Phi^{G/Ker(\alpha)}H\Z/2_*=\Phi^{G/Ker(\alpha)}(H\Z/2)^{Ker(\alpha)}_{*}
\r \Phi^G H\Z/2_*.
}
It is fairly obvious that for $n=1$ the spectral sequence associated with the ($1$-dimensional) cube \rref{egeom1}
collapses to $E^1$ and that in fact
\beg{egeommap2}{\Phi^{G/Ker\alpha}H\Z/2_*=\Z/2[x_\alpha]
}
where in the spectral sequence, the element $x_\alpha$ is filtration degree $1$ and is represented by the set
$\{(1)\}$ if we identify $G/Ker(\alpha)\cong \Z/2$. We will also denote the image under \rref{egeommap1}
$$f_\alpha(x_\alpha)\in \Phi^G H\Z/2$$
by $x_\alpha$.

\vspace{3mm}

\begin{theorem}
\label{t1}
We have
\beg{ephi1}{\begin{array}{l}\Phi^{G}_{*}H\Z/2 = \\
\Z/2[x_\alpha\mid \alpha\in G^*\smallsetminus \{0\}]/
(x_\alpha x_\beta+x_\alpha x_\gamma+x_\beta x_\gamma\mid \alpha+\beta+\gamma=0)
\end{array}
}
where the classes $x_\alpha$ are in dimension $1$.
\end{theorem}

\vspace{3mm}
Before proving the theorem, it is useful to record the following algebraic fact:

\begin{proposition}
\label{p1}
Let $\{\alpha_1,\dots,\alpha_k\}$ be a minimal $\Z/2$-linearly dependent subset of $G^*\smallsetminus \{0\}$, $k\geq 3$.
Then the ring $R_n$ on the right hand side of \rref{ephi1} satisfies
\beg{esymm1}{\sigma_{k-1}(x_{\alpha_1},\dots,x_{\alpha_k})=0.
}
(Here $\sigma_i$ denotes the $i$'th elementary symmetric polynomial.)
\end{proposition}

\begin{proof}
We will proceed by induction on $k$. For $k=3$, this is by definition. Suppose $k>3$ and suppose the statement is true
with $k$ replaced by $k-1$. Compute in $R_n$, denoting $\beta=\alpha_{k-1}+\alpha_k$:
\beg{ecomputrn}{\begin{array}{l}
\sigma_{k-1}(x_{\alpha_1},\dots,x_{\alpha_k})=\\
(x_{\alpha_k}+x_{\alpha_{k-1}})(x_{\alpha_1}\cdot\dots\cdot x_{\alpha_{k-2}})+x_{\alpha_k}x_{\alpha_{k-1}}
\sigma_{k-3}(x_{\alpha_1},\dots,x_{\alpha_{k-2}})=\\
(x_{\alpha_k}+x_{\alpha_{k-1}})(x_{\alpha_1}\cdot\dots\cdot x_{\alpha_{k-2}})
+(x_{\alpha_k}+x_{\alpha_{k-1}})x_\beta \sigma_{k-3}(x_{\alpha_1},\dots,x_{\alpha_{k-2}})=\\
(x_{\alpha_k}+x_{\alpha_{k-1}})\sigma_{k-2}(x_\beta,x_{\alpha_1},\dots,x_{\alpha_{k-2}})
\end{array}
}
Now $\{\beta,\alpha_1,\dots,\alpha_{k-2}\}$ is also a minimal linearly dependent set (note that
minimality is equivalent to the statement that $\alpha_1,\dots,\alpha_{k-1}$ are linearly independent
and $\alpha_1+\dots +\alpha_k=0$). Therefore, the right hand side of \rref{ecomputrn} is $0$ in $R_n$
by the induction hypothesis.
\end{proof}

\vspace{3mm}
\noindent
{\em Proof of Theorem \ref{t1}:}
We know that  $\Phi^{G}_{*}H\Z/2 $ is a ring, since $\Phi^G H\Z/2$ is an
$E_\infty$-ring spectrum.
By \rref{egeommap1}, we know that the $x_\alpha$'s represent elements of $\Phi^{G}_{*}H\Z/2 $, and hence polynomials in the $x_\alpha$'s do as well. Now it is important to note that \rref{egeom2aa} is not a spectral sequence of 
rings. However, there are maps arising from smashing $n$ cubes \rref{egeom1} (over $H\Z/2$) for $n=1$, and from this, it
is not difficult to deduce that for $S$ linearly independent, a monomial of the form
\beg{emonom1}{\prod_{s\in S}x_{s}^{r_s}, \; r_s\geq 1
}
is represented in \rref{egeom2aa} by
\beg{emonom2}{S\cdot \prod_{s\in S} x_{s}^{r_s-1}.}
(Note that by Lemma \ref{l1}, for $S$ not linearly independent, \rref{emonom2} does not survive to $E^2$.)
By Lemma \ref{l1}, we know that such elements generate the $E^2$-term as a $\Z/2$-module, so 
we have already proved that the spectral sequence associated with the cube \rref{egeom1} collapses to $E^2$.

Now counting basis elements in filtration degree $2$ 
shows that $\Phi^{G}_{*}H\Z/2 $ must have a quadratic relation
among the elements $x_\alpha$, $x_\beta$, $x_\gamma$ when
$$\alpha+\beta+\gamma=0.$$
(It suffices to consider $n=2$.)
The relation must be symmetric and homogeneous for reasons of dimensions, so the possible candidates 
for the relation are
\beg{erel1}{x_\alpha x_\beta +x_\alpha x_\gamma +x_\beta x_\gamma=0}
or
\beg{erel1alt}{x_\alpha x_\beta +x_\alpha x_\gamma +x_\beta x_\gamma+x_{\alpha}^2+x_{\beta}^{2}
+x_{\gamma}^{2}=0.}
We will prove the theorem by finding a basis of the monomials \rref{emonom1} of the ring on the right hand side
of \rref{ephi1} and matching them,
in the form \rref{emonom2}, with the $E^2$-term \rref{egeom2aa}. 

Before determining which of the relations \rref{erel1}, \rref{erel1alt} is correct, we observe (by induction)
that the ring $R_n$ given by the relation \rref{erel1} satisfies (identifying $G^*\cong (\Z/2)^n$)
\beg{egen1}{R_n=R_{n-1}\otimes \Z/2[x_{(0,\dots,0,1)}] +
\sum_{\alpha\in ((\Z/2)^{n-1}\smallsetminus\{0\})\times\{1\}} R_{n-1}\otimes x_\alpha\cdot\Z/2[x_\alpha]
}
and that the ring $R_{n}^{\prime}$ obtained from the relations \rref{erel1alt} satisfies a completely analogous
statement with $R_i$ replaced by $R_{i}^{\prime}$. By the identification between \rref{emonom1} and
\rref{emonom2}, we see that we obtain a $\Z/2$-module of the same rank as the $E^2$ term of the spectral
sequence of \rref{egeom1} in each dimension if and only if the sum \rref{egen1} for each $n$ is a direct sum
(and similarly for the case of $R_{n}^\prime$). Since we already know that the spectral sequence collapses to $E_2$,
we know that this direct sum must occur for whichever relation \rref{erel1} or \rref{erel1alt} is correct, and also that 
the ``winning" relation \rref{erel1} (resp. \rref{erel1alt}), ranging over all applicable choices of $\alpha$, 
$\beta$ and $\gamma$ generates all the relations in $\Phi^{G}_{*}H\Z/2 $.

We will complete the proof by showing that \rref{erel1alt} generates a spurious relation, and hence is eliminated.
This cannot be done for $n=2$, as we actually have $R_2\cong R_{2}^{\prime}$ via the (non-functorial 
isomorphism) replacing the generators 
$x_\alpha, x_\beta,x_\gamma$ with $x_\alpha+x_\beta$, $x_\alpha+x_\gamma$,
$x_\beta+x_\gamma$.

We therefore must resort to $n=3$. Let $\alpha_1=(1,0,0)$, $\alpha_2=(0,1,0)$, $\alpha_3=(0,0,1)$,
$\alpha_4=(1,1,1)$. Applying the computation \rref{ecomputrn} in the proof of Proposition \ref{p1} to
compute $\sigma_3(x_{\alpha_1},x_{\alpha_2},x_{\alpha_3},x_{\alpha_4})$ in the ring $R_{3}^{\prime}$,
we obtain
$$\begin{array}{l}\sigma_3(x_{\alpha_1},x_{\alpha_2},x_{\alpha_3},x_{\alpha_4})=\\
(x_{\alpha_1}+x_{\alpha_2})(x_{\alpha_3}^{2}+x_{\alpha_4}^{2}+x_{\beta}^2)
+(x_{\alpha_3}+x_{\alpha_4})(x_{\alpha_1}^{2}+x_{\alpha_2}^{2}+x_{\beta}^{2}).
\end{array}$$
As this is clearly not symmetrical in $x_{\alpha_1}$, $x_{\alpha_2}$, $x_{\alpha_3}$, $x_{\alpha_4}$, 
by permuting (say, using a $4$-cycle) and adding both relations, we obtain a spurious relation 
in dimension $3$ and filtration degree $2$, which shows that the analog of \rref{egen1} with $R_i$ replaced by
$R_{i}^{\prime}$ fails to be a direct sum for $n=3$, thereby excluding the relation \rref{erel1alt}, and completing the 
proof.
\qed

\vspace{3mm}

From the fact that \rref{egen1} is a direct sum, we obtain the following

\begin{corollary}
\label{cor1}
The Poincare series of the ring $R_n$ is
$$\frac{1}{(1-x)^n}\prod_{i=1}^{n}(1+(2^{i-1}-1)x).$$
\end{corollary}

\qed

\section{The coefficients of $H\Z/2$ suspended by a $G$-representation}\label{s4}

In this section, we will compute explicitly the coefficients of $H\Z/2$ suspended by 
\beg{erefv}{V=\sum_{\alpha\in G^*\smallsetminus \{0\}} m_\alpha \alpha}
with $m_\alpha\geq 0$. 

\vspace{3mm}
\begin{theorem}
\label{t2}
1. For $m_\alpha\geq 0$, $G^*\cong (\Z/2)^n$, recalling \rref{erefv}, the Poincare series of
$${\Sigma}^{V}H\Z/2_*$$
is
\beg{eserr1}{
\frac{1}{(1-x)^n}\left(
\sum_{(\Z/2)^k\cong H\subseteq G^*}
(-1)^k\left(\prod_{i=1}^{n-k}(1+(2^{i-1}-1)x)\right)
x^{\displaystyle k+\sum_{\alpha\in H\smallsetminus \{0\}}m_\alpha}\right).
}
2. For $m_\alpha\geq 0$, the canonical map 
$${\Sigma}^{V}H\Z/2\r\widetilde{E\mathcal{F}[G]}
\wedge H\Z/2
$$
(given by the smash product
of the inclusions $S^{m_\alpha \alpha}\r S^{\infty \alpha}$) induces an injective map on $\Z$-graded
homotopy groups.
\end{theorem}

\vspace{3mm}
We need the following purely combinatorial result. Let 
$$\left[\begin{array}{c}n\\k\end{array}
\right]=\frac{(2^n-1)\cdot (2^{n-1}-1)\cdot\dots\cdot (2^{n-k+1}-1)}{
(2^k-1)\cdot (2^{k-1}-1)\cdot \dots\cdot (2^1-1)}.$$
Note that this is the number of $k$-dimensional $\Z/2$-vector subspaces of $(\Z/2)^n$. The following statement amounts to part 1. of Theorem \ref{t2} for $m_\alpha=0$.

\vspace{3mm}

\begin{lemma}
\label{lcombo}
We have
$$
\sum_{k=0}^{n}(-1)^k \left[\begin{array}{c}n\\k\end{array}
\right] x^k\prod_{i=1}^{n-k}(1+(2^{i-1}-1)x)=(1-x)^n.
$$
\end{lemma}

\begin{proof}
Induction on $n$. We have
$$
\left[\begin{array}{c}n\\k\end{array}
\right]=
\left[\begin{array}{c}n-1\\k\end{array}
\right]+2^{n-k}
\left[\begin{array}{c}n-1\\k-1\end{array}
\right],
$$ 
so by the induction hypothesis,
$$\begin{array}{l}
\displaystyle\sum_{k=0}^{n}(-1)^k \left[\begin{array}{c}n\\k\end{array}
\right] x^k\prod_{i=1}^{n-k}(1+(2^{i-1}-1)x)=\\[3ex]
\displaystyle\sum_{k=0}^{n}(-1)^k \left(\left[\begin{array}{c}n-1\\k\end{array}
\right]+2^{n-k}
\left[\begin{array}{c}n-1\\k-1\end{array}
\right]\right) x^k\prod_{i=1}^{n-k}(1+(2^{i-1}-1)x).
\end{array}
$$
Splitting the right hand side into two sums, we get
$$\begin{array}{l}
\displaystyle\sum_{k=0}^{n-1}(-1)^k \left[\begin{array}{c}n-1\\k\end{array}
\right] x^k\prod_{i=1}^{n-k}(1+(2^{i-1}-1)x)+\\[3ex]
\displaystyle\sum_{k=1}^{n}(-1)^k2^{n-k} \left[\begin{array}{c}n-1\\k-1\end{array}
\right] x^k\prod_{i=1}^{n-k}(1+(2^{i-1}-1)x)=\\[3ex]
\displaystyle
(1-x)^n+\sum_{k=0}^{n-1}(-1)^k \left[\begin{array}{c}n-1\\k\end{array}
\right] x^k\prod_{i=1}^{n-k-1}(1+(2^{i-1}-1)x)2^{n-k-1}+\\[3ex]
\displaystyle\sum_{k=1}^{n}(-1)^k2^{n-k} \left[\begin{array}{c}n-1\\k-1\end{array}
\right] x^k\prod_{i=1}^{n-k}(1+(2^{i-1}-1)x)=(1-x)^n.
\end{array}
$$
\end{proof}

\vspace{3mm}
\noindent
{\em Proof of Theorem \ref{t2}:}
We will proceed by induction on $n$. Assume 1. and 2. are true for lower values of $n$. Then, for the 
given $n$, we proceed by induction on
$$\ell=|\{\alpha\in G^*\mid m_\alpha>0\}|.$$
For $\ell=0$, 1. follows from Lemma \ref{lcombo} and 2. is obvious (by ring structure of
$\Phi^GH\Z/2$). Suppose $\ell\geq 1$ and 1., 2. are true for lower values of $\ell$. Setting
\beg{erefvl}{V_\ell=\sum_{i=1}^{\ell} m_{\alpha_i} \alpha_i,}
we will study the
effect on coefficients $(?)_*$ of the cofibration sequence
\beg{e1t2}{\diagram S(m_\ell\alpha_\ell)_+\wedge {\Sigma}^{V_{\ell-1}}H\Z/2\dto\\
 {\Sigma}^{V_{\ell-1}}H\Z/2\dto\\
 {\Sigma}^{V_\ell}H\Z/2.
\enddiagram
}
First, we observed that the first map factors through the top row of the diagram
\beg{e2t2}{\diagram
(EG/Ker \alpha_\ell)_+\wedge {\Sigma}^{V_{\ell-1}}H\Z/2\rto\dto
&
{\Sigma}^{V_{\ell-1}}H\Z/2\dto\\
(EG/Ker\alpha_\ell)_+\wedge \widetilde{E\mathcal{F}[G]}\wedge H\Z/2\rto&
\widetilde{E\mathcal{F}[G]}\wedge H\Z/2.
\enddiagram
}
Next, the right column of \rref{e2t2} is injective on $(?)_*$ by 2. for $\ell-1$, and hence the top row, 
and hence also the first map \rref{e1t2}, is $0$ on $(?)_*$.

Now the Poincare series of 
\beg{e3t2}{(S(m_\ell\alpha_\ell)_+\wedge {\Sigma}^{V_{\ell-1}}H\Z/2)^{G}_{*}
}
is
$$\frac{1-x^{m_\ell}}{1-x}$$
times the Poincare series of
\beg{e4t2}{
({\Sigma}^{V_{\ell-1}}H\Z/2)^{Ker\alpha_\ell}_{*},
}
which, when multiplied by $x$ and added to the Poincare series of
$$( {\Sigma}^{V_{\ell-1}}H\Z/2)^{G}_{*},$$
is \rref{eserr1} by the induction hypothesis. This proves 1.

To prove 2., we observe that the elements of \rref{e3t2} are generated by powers of $x_{\alpha_\ell}$
multiplied by elements of \rref{e4t2}, so again, we are done by the induction hypothesis.
\qed

\vspace{3mm}
\section{The complex calculating $RO(G)$-graded coefficients}\label{s5}

To calculate the $RO(G)$-graded coefficients of $H\Z/2_G$ in dimensions given by virtual representations,
we introduce another spectral sequence. In fact, we will again use the cofibration sequence \rref{egeom++}, but 
we will rewrite it as
\beg{egeom+++}{S^0\r S^{\infty\alpha}\r \Sigma EG/Ker(\alpha)_+.}
We will smash the second maps of \rref{egeom+++} over all $\alpha\in G^*\smallsetminus\{0\}$, to obtain a cube
\beg{egeom**}{\bigwedge_{\alpha\in G^*\smallsetminus \{0\}}(S^{\infty\alpha}\r \Sigma EG/Ker(\alpha)_+)
}
whose iterated fiber is $S^0$.
Our method is to smash with $H\Z/2_G$ and take $RO(G)$-graded coefficients:
\beg{egeom**1}{(\bigwedge_{\alpha\in G^*\smallsetminus \{0\}}(S^{\infty\alpha}\r \Sigma EG/Ker(\alpha)_+)\wedge
H\Z/2)_\star,
}
thus yielding a spectral sequence calculating $H\Z/2_\star$.

However, there is a key point to notice which drastically simplifies this calculation. Namely,  smashing \rref{egeom++}
with $EG/Ker(\alpha)+$, the first morphism becomes an equivalence, thus showing that
\beg{ezero1}{EG/Ker(\alpha)_+\wedge S^{\infty\alpha} \simeq *.
}
Together with \rref{egeom2}, this shows that the only vertices of the cube \rref{egeom**} which are
non-zero are actually those of the form where all the $\alpha$'s for which we take the term $S^{\infty\alpha}$
in \rref{egeom**} are those {\em not vanishing} on some subgroup $A\subseteq G$, while those $\alpha$'s 
for which we take the term $\Sigma EG/Ker(\alpha)_+$ are those non-zero elements of $G^*$ which
{\em do vanish} on $A$, i.e. non-zero elements of $(G/A)^*$. The corresponding vertex of \rref{egeom**} is
then a suspension of
\beg{evert1}{gr_A(S^0)=EG/A_+\wedge \widetilde{E\mathcal{F}[A]}.
}
We also put 
$$gr_A(H\Z/2)=gr_A(S^0)\wedge H\Z/2.$$
Because of the high number of zero terms, the spectral sequence may be regraded by $rank_{\Z/2}(A)$, thus having 
only $n$, instead of $2^n-1$, filtration degrees. (Note that the cube \rref{egeom**} 
may be reinterpreted as a ``filtration" of the spectrum $S^0$; from this point of view, we have simply observed
that many of the filtered parts coincide.)

It is now important, however, to discuss the grading seriously. Since we index coefficients homologically, 
we will write the spectral sequence in homological indexing. Additionally, we want the term $gr_G(S^0)$ be in 
filtration degree $0$ (since that is where the unit is). Thus, the (homologically indexed) filtration degree of \rref{evert1} is
$$p=rank(A)-n$$
(a non-positive number). Thus, 
$$\pi_k({\sum}^{\sum_{\alpha\in G^*\smallsetminus \{0\}}m_\alpha \alpha}gr_A(H\Z/2)\subseteq 
E^{1}_{rank(A)-n,k+n-rank(A)}
$$
or, put differently, for a given choice of the $m_\alpha$'s, 
\beg{estare1}{E^{1}_{p,q}=\bigoplus_{rank(A)=n+p}\pi_{q+p-\sum m_\alpha \alpha}gr_A(H\Z/2),\; p=-n,\dots,0.
}
We will next describe explicitly the differential
\beg{estard1}{d^1:E^{1}_{p,q}\r E^{1}_{p-1,q}.
}
Let us first introduce some notation. To this end, we need to start out by describing the $E^1$-term more 
explicitly. 

In effect, we can calculate $gr_A(H\Z/2)_\star$ by taking first the $A$-fixed points using 
Theorem \ref{t1} with $G$ replaced by $A$, and then applying the Borel homology spectral sequence
for $G/A$. This spectral sequence collapses because there exists a splitting
\beg{esplitting}{\diagram
A\rto^\subseteq\drto_= & G\dto^r\\
&A.
\enddiagram
}
However, the splitting is not canonical, and this is reflected by the choice of generators we observe. More explicitly,
the splitting determines for each representation
$$0\neq \beta:A\r \Z/2$$
an extension
$$\widetilde{\beta}:G\r\Z/2.$$
One difficulty with describing Borel homology is that it does not naturally form a ring. Because of that, it is more convenient to describe first the coefficients of
\beg{egamma1}{\gamma_{A}(H\Z/2):=F(EG/A_+,\widetilde{E\mathcal{F}[A]})\wedge H\Z/2.
}
This is an ($E_\infty$-) ring spectrum, and its ring of coefficients is given by
\beg{estargamma}{\begin{array}{l}\gamma_A(H\Z/2)_\star=\\
(\Z/2[x_{\widetilde{\beta}},u^{\pm1}_{\widetilde{\beta}}, 
u^{\pm1}_{\widetilde{\beta}+\alpha}
\mid \beta\in A^*\smallsetminus \{0\},\; \alpha\in (G/A)^*\smallsetminus \{0\}]/\\
(x_{\widetilde{\alpha}}x_{\widetilde{\beta}}+x_{\widetilde{\alpha}}x_{\widetilde{\gamma}}
+x_{\widetilde{\beta}}x_{\widetilde{\gamma}}\mid \alpha+\beta+\gamma=0))
[(y_\alpha u^{-1}_{\alpha})^{\pm 1}]\\
{}[[y_\alpha\mid
\alpha\in (G/A)^*\smallsetminus \{0\}]]/
(y_{\alpha+\alpha^\prime}-y_{\alpha}-y_{\alpha^\prime})
\end{array}
}
where the $RO(G)$-graded dimensions of the generators are
$$|u_\gamma|=-\gamma,\; |x_\gamma|=1, \;|y_\gamma|=-1.$$
We may then describe $gr_A(H\Z/2)_\star$ as the $dim_{\Z/2}(G/A)$'th (=only non-trivial)
local cohomology module of the ring $\gamma_A(H\Z/2)$ with respect to the ideal generated by the $y_\alpha$'s.
Note that after taking $A$-fixed points first, this is the usual computation of $G/A$-Borel homology from
the corresponding Borel cohomology. Recall that $H^*_I(R)$ for a finitely generated ideal $I$ of a commutative ring $R$
is obtained by choosing finitely many generators $y_1,\dots, y_\ell$ of $I$, tensoring, over $R$, the cochain
complexes
$$R\r y_i^{-1}R$$
(with $R$ in degree $0$) and taking cohomology. It is, canonically, independent of the choice of generators. 
In the present case, we are simply dealing with the power series ring $R$ in $dim_{\Z/2}(G/A)$ generators
over a $\Z/2$-algebra, and the augmentation ideal. Taking
the defining generators of the power series ring, we see immediately that only the top local cohomology group
survives.

We note that the basic philosophy of our notation is
\beg{ephil}{``y_\alpha = x_{\alpha}^{-1}".}
As a first demonstration of this philosophy, let us investigate the effect of a change of the splitting \rref{esplitting}.
Writing metaphorically
\beg{emeta1}{x_{\widetilde{\beta}+\alpha}x_{\widetilde{\beta}}+
x_{\widetilde{\beta}+\alpha}x_{{\alpha}}+x_{\widetilde{\beta}} x_\alpha=0,
}
we get
\beg{emeta2}{x_{\widetilde{\beta}+\alpha}x_{\widetilde{\beta}}y_\alpha+
x_{\widetilde{\beta}+\alpha}+x_{\widetilde{\beta}}=0,
}
from which we calculate
\beg{emeta3}{x_{\widetilde{\beta}+\alpha}=x_{\widetilde{\beta}}(1+x_{\widetilde{\beta}}y_\alpha)^{-1}
=\sum_{k=0}^{\infty} x_{\widetilde{\beta}}^{k+1}y_{\alpha}^{k}.
}
This formula is correct in $\gamma_A(H\Z/2)_\star$ and hence can also be used in the module $gr_A(H\Z/2)_\star$.

Next, we will describe the $d^1$ of \rref{estare1}. These connecting maps will be the sums of maps of degree $-1$
of the form
\beg{emetad1}{d^{AB}:gr_A(H\Z/2)_\star\r  gr_B(H\Z/2)_\star
}
where $B\subset A$ is a subgroup with quotient isomorphic to $\Z/2$. Let $\beta:A\r\Z/2$ be the unique
non-trivial representation which vanishes when restricted to $A$. The key point is to observe that the 
canonical map
\beg{emetad11}{\diagram
EG/A_+\wedge \widetilde{E\mathcal{F}[B]}\wedge S^{\infty\widetilde{\beta}}
\rto^(.55)\sim 
&
EG/A_+\wedge \widetilde{E\mathcal{F}[A]}
\enddiagram
}
is an equivalence, and hence \rref{emetad1} can be calculated by smashing with $H\Z/2$ the connecting
map
\beg{emetad12}{EG/A_+\wedge \widetilde{E\mathcal{F}[B]}\wedge S^{\infty\widetilde{\beta}}
\r \Sigma EG/B_+\wedge \widetilde{E\mathcal{F}[B]}.
}
Consequently, \rref{emetad1} is a homomorphism of $\gamma_AH\Z/2_\star$-modules, and is computed, just like
in dimension $1$, by replacing
$$x_{\widetilde{\beta}}\mapsto y_{\widetilde{\beta}}^{-1}$$
and multiplying by $y_{\widetilde{\beta}}$. (Note that independence of the splitting $\widetilde{\beta}$ at this point
follows from topology; it is a non-trivial fact to verify purely algebraically.

\vspace{3mm}
We have thereby finished describing the $d^1$ of the spectral sequence 
\rref{estare1}. The main result of the present section is the following 

\vspace{3mm}

\begin{theorem}
\label{t3}
The spectral sequence \rref{estare1} collapses to $E^2$. 
\end{theorem}

\vspace{3mm}
We will first prove some auxiliary results. 

\vspace{3mm}

\begin{lemma}
\label{t3l1}
The Borel homology spectral sequence of any cell $H\Z/2_G$-module with 
cells
\beg{et3+}{\Sigma^?G_+\wedge H\Z/2
}
collapses to $E^2$.
\end{lemma}

\begin{proof}
Taking $G$-fixed point, we obtain a cell $H\Z/2$-module with one cell for each cell \rref{et3+}. Now the homotopy 
category of $H\Z/2$-modules is equivalent to the derived category of $\mathbb{F}_2$-vector spaces, and a 
chain complex of 
$\mathbb{F}_2$-modules is isomorphic to a sum of an acyclic module and suspensions of $\mathbb{F}_2$.
\end{proof}

\vspace{3mm}

\begin{lemma}
\label{ltensor}
Let $G$,$H$ be finite groups and let $X$ be an $G$-cell spectrum and let $Y$ be an $H$-cell spectrum 
(all indexed over the complete universe). 
Then
$$(H\Z/2_{G\times H} \wedge i_\sharp X\wedge j_\sharp Y)^{G\times H} \simeq 
(H\Z/2_G\wedge X)^G\wedge_{H\Z/2} (H\Z/2_H\wedge Y)^H.$$
Here on the left hand side, $i_\sharp$ is the functor introducing trivial $H$-action on a $G$-spectrum and pushing forward
to the complete universe, while $j_\sharp$ is the functor introducing trivial $G$-action on an $H$-spectrum and
pushing forward to the complete universe.
\end{lemma}

\begin{proof}
First consider $Y=S^0$. Then we have the forgetful map
$$(i_\sharp X\wedge H\Z/2)^{G\times H}\r (X\wedge H\Z/2)^G$$
which is an equivalence because it is true on cells.

In general, we have a map
$$Z^\Gamma\wedge T^\Gamma \r (Z\wedge T)^\Gamma,$$
so take the composition
$$
\begin{array}{l}
(X\wedge H\Z/2)^G\wedge (Y\wedge H\Z/2)^H=\\
(i_\sharp X\wedge H\Z/2)^{G\times H}\wedge (j_\sharp Y\wedge H\Z/2)^{G\times H}\r\\
(i_\sharp H\Z/2\wedge \j_\sharp Y\wedge H\Z/2)^{G\times H}\r\\
(i_\sharp X\wedge j_\sharp Y)^{G\times H}
\end{array}$$
(the last map coming from the ring structure on $H\Z/2$). Then again this map is an equivalence on cells, 
and hence an equivalence.
\end{proof}

\vspace{3mm}

\begin{lemma}
\label{t3l2}
Recalling again the notation \rref{erefv},
(a) the spectral sequence \rref{egeom**1} for 
\beg{et3l21*}{\pi_*{\Sigma}^{V}H\Z/2}
with all $m_\alpha\geq 0$ collapses to the $E^2$-term in filtration degree $0$.

(b) Let $m_\alpha\leq 0$ for all $\alpha$ and let 
$$S=\{\alpha\in G^*\smallsetminus\{0\}\mid m_\alpha\neq 0\}.$$ 
Suppose the subgroup of $G^*$ spanned by $S$ has rank $m$. Then the spectral sequence \rref{egeom**1} for 
\rref{et3l21*} collapses to $E^2$ in filtration degree $-m$.
\end{lemma}

\begin{proof}
Recall the notation \rref{erefvl}.
Let $G^*\smallsetminus \{0\}=\{\alpha_1,\dots,\alpha_{2^n-1}\}$. When $\alpha_k$ is linearly independent
of $\alpha_1,\dots,\alpha_{k-1}$, we have
\beg{et3l21}{
\begin{array}{l}
\displaystyle
\pi_*{\Sigma}^{V_k}H\Z/2\cong\\[4ex]
\displaystyle
\pi_*\left({\Sigma}^{V_{k-1}}H\Z/2\right)^G
\otimes
\pi_*(\Sigma^{m_{\alpha_k} \gamma} H\Z/2)^{\Z/2}
\end{array}
}
where $\gamma$ is the sign representation of $\Z/2$m by Lemma \ref{ltensor}. Note that in the case (b), we may, without loss of generality, assume $m=n$ (i.e. that $S$ spans $G^*$) and that what we just said occurs for $k=1,\dots,n$
and additionally that $m_{\alpha_i}<0$ for $i=1,\dots,n$.

When $\alpha_k$ is a linear combination of $\alpha_1,\dots,\alpha_{k-1}$, and $m_{\alpha_k}\neq 0$,
we use the cofibration sequence
\beg{et3l22}{\begin{array}{l}\displaystyle S(m_{\alpha_k}\alpha_k)_+\wedge 
{\Sigma}^{V_{k-1}}H\Z/2\r\\[4ex]\displaystyle
{\Sigma}^{V_{k-1}}H\Z/2\r {\Sigma}^{
V_k}H\Z/2
\end{array}
}
in the case (a) and
\beg{et3l23}{\begin{array}{l}\displaystyle 
{\Sigma}^{V_k}H\Z/2\r\\[4ex]\displaystyle
{\Sigma}^{V_{k-1}}H\Z/2\r DS(-m_{\alpha_k}\alpha_k)_+\wedge {\Sigma}^{
V_{k-1}}H\Z/2
\end{array}
}
in the case (b). If we denote each of these cofibration sequences symbolically as
$$A\r B\r C,$$
then in the case (a), \rref{et3l22} gives a short exact sequence of the form
\beg{et3l22a}{0\r E^1 A\r E^1 B\r E^1C\r 0 
}
of the spectral sequence of \rref{egeom**1} where in the $A$-term, we replace $G$ by $Ker (\alpha_k)$
and $H\Z/2$ by $S(m_k\alpha_k)_+\wedge H\Z/2$. By the induction hypothesis, however, the homology of 
$E^1A$ is concentrated in the top filtration degree, which is $-1$ from the point of view of $G$, and the 
homology of $E^1B$ is concentrated in filtration degree $0$, so the long exact sequence in homology 
gives
\beg{et3l22b}{0\r E^2\r E^2C\r \Sigma E^2A\r 0
}
which is all in filtration degree $0$, so our statement follows.

In the case (b), by our assumptions, we have $k>n$. Additionally, \rref{et3l23} gives a short exact
sequence
\beg{et3l23a}{0\r \Sigma^{-1}E^1C\r E^1 A\r E^1 B\r 0,
}
but by the induction hypothesis (using the fact that a set of generators of $G^*$ projects to a set of generators
of the factor group $Ker(\alpha_k)^*$), the homology of the first and last term is concentrated in 
filtration degree $-n$, so \rref{et3l23a} translates to the same short exact sequence with $E^1$
replaced by $E^2$, which is entirely in filtration degree $-n$, and the statement follows.
\end{proof}

\vspace{3mm}

To continue the proof of Theorem \ref{t3}, let again 
$$G^*\smallsetminus \{0\}=\{\alpha_1,\dots,\alpha_{2^n-1}\}.$$
Consider
\beg{et3p*}{{\Sigma}^{V_{2^n-1}} H\Z/2
}
and let, this time, without loss of generality, 
$$m_{\alpha_1},\dots,m_{\alpha_q}<0,$$ 
$$m_{\alpha_{q+1}},\dots,m_{\alpha_{2^n-1}}\geq 0.$$
Let $A=Ker(\alpha_1)\cap\dots\cap Ker(\alpha_q).$ We will consider the sequence of cofibrations
\rref{et3l22} with $q\leq k< 2^n-1$. Resolving this recursively, we may consider this as a cell object 
construction in the category of $H\Z/2_G$-modules, with "cells" of the form of suspensions (by an integer) of
\beg{et3p+}{\begin{array}{l}
\displaystyle G/(Ker(\alpha_{j_1})\cap\dots\cap Ker(\alpha_{j_p})_+\wedge
{\Sigma}^{V_q} H\Z/2,\\[2ex]
 q<j_1<\dots <j_p\leq 2^n-1.
\end{array}
}
By the {\em degree} of a cell $c$, we shall mean the number
$$deg(c)=n-rank(Ker(\alpha_{j_1})\cap\dots\cap Ker(\alpha_{j_p})),
$$
and by the {\em $A$-relative degree} of $c$, we shall mean
$$\begin{array}{l}deg_A(c)=rank(G/A)-\\
rank(Ker(\alpha_{j_1})\cap\dots\cap Ker(\alpha_{j_p})/Ker(\alpha_{j_1})\cap\dots\cap 
Ker(\alpha_{j_p})\cap A).
\end{array}$$
We see easily from the construction that cells of a given degree are attached to cells of strictly lower degree,
and that cells of a given $A$-relative degree are attached to cells of lesser or equal $A$-relative degree. 
(Roughly speaking, "more free" cells are attached to "less free" ones.)

\vspace{3mm}

\begin{lemma}
\label{t3l3}
The spectral sequence arising from the cube \rref{egeom**1} with $H\Z/2$ replaced by the complex
formed by our "cells" of $A$-relative degree $d$ collapses to $E^2$ concentrated in filtration degree
$d-rank(G/A)$. 
\end{lemma}

\begin{proof}
Within a given $A$-relative degree $d$, attaching cells of each consecutive degree results in a short exact sequence of the form \rref{et3l22a} where the first two terms collapse to $E^2$ in filtration degree $d-rank(G/A)-1$ and
$d-rank(G/A)$, respectively. Thus, there results a short exact sequence of the form 
\rref{et3l22b} in filtration degree $d-rank(G/A)$, as claimed.
\end{proof}

\vspace{3mm}
\noindent
{\em (The rest of) the proof of Theorem \ref{t3}:}
Filtering cells of \rref{et3p*} by $A$-relative degree, we obtain a spectral sequence $\mathcal{E}$ converging
to $E^2$ of the spectral sequence of the cube \rref{egeom**1} for \rref{et3p*}. By Lemma \ref{t3l3},
all the terms will be of the same \rref{egeom**1}-filtration degree $-rank(G/A)$, which is the complementary
degree of $\mathcal{E}$. (Note that in this discussion, we completely ignore the original topological degree.)
Thus, being concentrated in one complementary degree, $\mathcal{E}$ collapses to $E^2$ in that complementary
degree.

However, by precisely the same arguments, we can write a variant $\widetilde{\mathcal{E}}$
of the spectral sequence $\mathcal{E}$ in homotopy groups (rather than \rref{egeom**1} $E^1$-terms) 
of the filtered pieces of 
\rref{et3p*} by $A$-relative degree. By Lemma \ref{t3l3}, $\widetilde{\mathcal{E}}^1\cong \mathcal{E}^1$,
and $d_{\widetilde{\mathcal{E}}}^{1}$, $d_{\mathcal{E}}^{1}$ have the same rank (since they are computed by
the same formula). It follows that $\widetilde{\mathcal{E}}^2\cong \mathcal{E}^2$, both collapsing to a single
complementary degree. Therefore, it follows that $E^2$ (of the spectral sequence associated with \rref{egeom**1} for
\rref{et3p*}) is isomorphic to the homotopy of \rref{et3p*}, and hence the spectral sequence collapses to $E^2$
by a counting argument.
\qed

\vspace{3mm}
\section{Example: $n=2$}\label{s6}

In the case $n=2$, there are only three sign representations $\alpha$, $\beta$, $\gamma$ which play a symmetrical
role and satisfy
\beg{eex+}{\alpha+\beta+\gamma=0\in G^*,
}
which means that the Poincare series of the homotopy 
\beg{eex*}{\pi_*(\Sigma^{k\alpha+\ell\beta+m\gamma}H\Z/2) 
}
can be written down explicitly.

First recall that by Theorem \ref{t2}, for $k,\ell,m\geq 0$, the Poincare series is
\beg{eex1}{\frac{1}{(1-x)^2}(1+x-x^{1+k}-x^{1+\ell}-x^{1+m}+x^{2+k+\ell+m}).
}
If $k,\ell<0,m\leq 0$, by the proof of Lemma \ref{t3l2}, the formula \rref{eex1} is still valid when multiplied by
$x^{-2}$ (since all the homotopy classes are in filtration degree $-2$).

If $k,\ell<0,m>0$, in the proof of Theorem \ref{t3}, $A=0$, so the $A$-relative degree and the degree
coincide. Further, by \rref{eex+} and our formula for the differential $d^1$ of the spectral sequence of
\rref{egeom**1}, the differential $d^{1}_{\mathcal{E}}$ has maximal possible rank (i.e. "everything that can
cancel dimension-wise will"). We conclude that the $E^2$ is concentrated in filtration degrees $-1$ and $-2$. 
By the cancellation principle we just mentioned, the Poincare series can still be recovered from the formula
\rref{eex1}. If we write the expression \rref{eex1} as
\beg{eex**}{P_+(x)-P_-(x)
}
where $P_+(x)$ (resp. $-P_-(x)$) is the sum of monomial summands with a positive coefficient (resp. with a negative
coefficient) then the correct Poincare series in this case is
$$x^{-2}P_+(x)+x^{-1}P_-(x),$$
the two summands of which  represent classes in filtration degree $-2$ and $-1$, respectively. 

Similarly, one shows that if $k,\ell \geq 0$, $m<0$, the $E^2$ collapses to filtration degrees $0$ and $-1$,
and the Poincare series in this case is
$$P_+(x)+x^{-1}P_-(x).$$
All other cases are related to these by a symmetry of $(\Z/2)^2$.

\vspace{3mm}
\noindent
{\bf Remark:} It might  seem natural to conjecture that the classes of different filtration degrees in $E^2$ may be
of different dimensions, with a gap between them (evoking the "gap condition" which was proved for $\Z/2$ in
\cite{hk}, and made famous for the group $\Z/8$ by the Hill-Hopkins-Ravenel \cite{hhr} work on the Kervaire invariant $1$
problem). However, one easily sees that for $n\geq 3$, classes of different filtration degrees may occur in 
the same dimension. For example, by Lemma \ref{ltensor} and by what we just proved, such a situation always occurs
for $\pi_*\Sigma^{4\alpha+4\beta-2\gamma+4\delta}H\Z/2$ where $\alpha, \beta,\gamma$ are the three
sign representations of $\Z/2\times\Z/2\times \Z/2$ factoring through the projections to the first two copies of
$\Z/2$,
and $\delta$ is the sign representation which factors through the projection onto the last $\Z/2$.


\vspace{10mm}
\begin{thebibliography}{99}

\bibitem{abk} W. Abram and I. Kriz: The equivariant complex cobordism ring of a finite abelian group, preprint, 2012

\bibitem{araki} S. Araki: Orientations in t-cohomology theories, {\em Japan J. Math} 16 (1) (1978) 363-416

\bibitem{atiyah} M.F. Atiyah: K-Theory and Reality, {\em Quar. J. Math} (1966) 367-386

\bibitem{as} M.F.Atiyah, I.M.Singer:
The index of elliptic operators. I,
{\em Ann. of Math.} (2) 87 1968 484–53

\bibitem{bredon} G. Bredon: {\em Equivariant cohomology theories}, Springer Lecture Notes in Mathematics (1967), no. 34,

\bibitem{kar} P. Donovan,  M. Karoubi: Graded Brauer groups and K-theory with local
coefficients, {\em Publ. Math. IHES} 38 (1970) 5-25

\bibitem{dress} A.W.M. Dress: Notes on the theory of representations of finite groups. Part I: 
The Burnside ring of a finite group and some AGN-applications

\bibitem{fht} D.S.Freed, M.J.Hopkins, C.Teleman: Loop groups and twisted K-theory I, {\em J. Topol.} 
4 (2011), no. 4, 737-798

\bibitem{green} J.P.C.Greenlees: {\em Adams spectral sequences in equivariant topology}, Thesis, 
Cambridge University (1985)

\bibitem{gmsur}  J.P.C.Greenlees, J.P. May: Equivariant stable homotopy theory, {\em Handbook of algebraic topology}, 277-323, North-Holland, Amsterdam, 1995


\bibitem{hhr} M.Hill, M.J.Hopkins, D.Ravenel: On the non-existence of elements of Kervaire invariant one, arXiv:0908.3724 (2009)

\bibitem{hk} P. Hu, I. Kriz: Real-oriented homotopy theory and an analogue of the Adams-Novikov spectral sequence, 
{\em Topology} 40 (2001), no.2, 317-399

\bibitem{hkherm} P. Hu, I. Kriz, Topological Hermitian Cobordism, arXiv:1110.5608, to appear in 
{\em J. Homotopy Re. Str.}

\bibitem{hkem} P. Hu, I. Kriz, Coefficients of the constant Mackey functor over cyclic p-groups, Preprint, 2010
 
\bibitem{hko} P. Hu, I. Kriz, K. Ormsby: The homotopy limit problem for Hermitian K-theory, equivariant homotopy theory and motivic real cobordism, {\em Adv. Math.} 228 (2011), no. 1, 434-480


\bibitem{kriz} I.Kriz: The $\Z/p$-equivariant complex cobordism ring, Homotopy invariant algebraic structures, 
{\em Contemp. Math.}, 239, Amer. Math. Soc. (1999) 217-223

\bibitem{sk} S.Kriz: Equivariant cohomology and the super reciprocal plane of a hyperplane arrangement, preprint, 2015

\bibitem{land} P.Landweber:Conjugations on complex manifolds and equivariant homotopy of MU,
{\em Bull. AMS} 74 (1968) 271-274

\bibitem{sgl}  G.L.Lewis: The $RO(G)$-graded equivariant ordinary cohomology 
of complex projective spaces with linear $\Z/p$ actions, {\em Algebraic topology and transformation groups} 
(Göttingen, 1987), 53-122, Lecture Notes in Math., 1361, Springer, Berlin, 1988

\bibitem{lmm} G.Lewis, J.P.May, J. McClure:
Ordinary RO(G)-graded cohomology,
{\em Bull. Amer. Math. Soc.} (N.S.) 4 (1981), no. 2, 208-212

\bibitem{lms} L.G. Lewis, J.P. May, M. Steinberger, J.E. McClure: {\em Equivariant stable homotopy theory}, 
Lecture Notes in Mathematics, 1213 (1986)

\bibitem{ps} N. J. Proudfoot and D. Speyer: A broken circuit ring, {\em Beitr\"{a}ge Algebra Geom.}
47 (2006), no. 1, 161-166

\bibitem{terao} H.Terao:
Algebras generated by reciprocals of linear forms, 
{\em J. Algebra} 250 (2002), no. 2, 549–558

\bibitem{td} T. tom Dieck: Bordism of $G$-manifolds and integrability theorems, {\em Topology 9} (1970) 345-358

\end{thebibliography}
\end{document}